\documentclass{amsart}
\usepackage{amsmath,amsfonts,amssymb,amsthm,eucal,enumitem,graphicx,bm,color,verbatim,bbm}

\usepackage{hyperref}

\newtheorem{thm}{Theorem}

\newtheorem{lemma}[thm]{Lemma}
\newtheorem{prop}[thm]{Proposition}

\newtheorem*{prop*}{Proposition}

\newcommand{\Unitary}[1]{\mathbb{U}\left(#1\right)}
\newcommand{\SUnitary}[1]{\mathbb{SU}\left(#1\right)}

\newcommand{\Circle}{\mathbb{S}^1}

\newtheorem{hack}{Proposition}
\newcommand{\R}{\mathbb{R}}

\newcommand{\E}{\mathbb{E}}
\newcommand{\Prob}{\mathbb{P}}
\newcommand{\N}{\mathbb{N}}

\newcommand{\Z}{\mathbb{Z}}

\DeclareMathOperator{\tr}{tr}

\DeclareMathOperator{\vol}{vol}

\newcommand{\inprod}[2]{\left\langle #1, #2 \right\rangle}

\renewcommand{\P}{\mathbb{P}}

\allowdisplaybreaks[4]

\thanks{\footnotemark {$^\dagger$} Supported in part by NSF DMS 1612589.}
\thanks{\footnotemark {$^*$} Supported in part by NSF DMS 1255574.}
\email{ese3@case.edu}
\email{melcher@virginia.edu}

\begin{document}

\title[Convergence of the ESM of unitary BM]{Convergence of the empirical spectral measure of unitary Brownian motion}

\author[Meckes]{Elizabeth Meckes{$^\dagger$}}
\address[Elizabeth Meckes]{Department of Mathematics, Case Western Reserve University
220 Yost Hall, Cleveland, OH 44106 USA}

\author[Melcher]{Tai Melcher{$^*$}}
\address[Tai Melcher]{Department of Mathematics\\
University of Virginia\\
Charlottesville, Virginia 22904 USA}

\begin{abstract}
Let $\{U^N_t\}_{t\ge 0}$ be a standard Brownian motion on $\Unitary{N}$. For fixed $N\in\N$ and $t>0$, we give explicit bounds on the $L_1$-Wasserstein distance of
the empirical spectral measure of $U^N_t$ to both the ensemble-averaged
spectral measure and to the large-$N$ limiting measure identified by
Biane. We are then able to use these bounds to control the rate of convergence of paths of the measures on compact time intervals. The proofs use tools developed by the first author to study convergence rates of the classical random matrix ensembles, as well as recent estimates for the convergence of the moments of the ensemble-averaged spectral distribution.
\end{abstract}

\maketitle
\tableofcontents

\section{Introduction}
This paper studies the convergence of the empirical spectral measure
of Brownian motion on the unitary group $\Unitary{N}$ to its large $N$
limit. Brownian motion on large unitary groups has generated
significant interest in recent years, due in part to its relationships
with two-dimensional Yang-Mills theory and with the object from free
probability theory called free unitary Brownian motion. As is natural
in the context of random matrices, there has been particular focus on
the asymptotic behavior (as $N$ tends to infinity) of the spectral
measure of unitary Brownian motions; see for example \cite{Rains1997,
  Xu1997,Biane1997,Biane1997-2, Levy2008,LevyMaida2010,DHK2013,
  Kemp2015,CDK2016} and the references therein.

Of course, many tools have been developed to study the spectral
distributions of random matrices in high dimension in a variety of
contexts.  Among them is an approach developed by the first author
with M.~Meckes (see \cite{MM2016} for a survey) which allows for
quantitative estimates on rates of convergence of the empirical
spectral measure in a wide assortment of random matrix ensembles. This
approach is based on concentration of measure and bounds for suprema
of stochastic processes, in combination with more classical tools from
matrix analysis, approximation theory, and Fourier analysis. In the
present paper, we combine some of these techniques with recent
estimates on the rates of convergence of the moments for the empirical
spectral distribution of unitary Brownian motion \cite{CDK2016} to
prove asymptotically almost sure rates of convergence.  We then use these bounds to control the rate of convergence of paths of the measures on compact time intervals.

{\bf Statement of results. }
Let $\Unitary{N}$ denote the unitary group and $\mathfrak{u}(N)$ its Lie algebra of skew-Hermitian matrices equipped with the scaled (real) inner product $\inprod{U}{V}_N:=N\tr(UV^*)$.
This is the unique scaling that gives meaningful limiting behavior as $N\rightarrow\infty$; see for example Remark 3.4 of \cite{DHK2013}.
The inner product on $\mathfrak{u}(N)$ induces a left-invariant Riemannian metric on $\mathbb{U}(N)$, and we may define Brownian motion on $\Unitary{N}$ as the Markov diffusion $\{U^N_t\}_{t\ge0}$ issued from the identity with generator $\frac{1}{2}\Delta_{N}$, that is, one half the left-invariant Laplacian on $\mathbb{U}(N)$ with respect to this metric. One may equivalently describe $U^N_t$ as the solution to the It\^o stochastic differential equation
\begin{equation*}
dU^N_t=U^{N}_tdW^N_t-\frac{1}{2}U^{N}_tdt
\end{equation*}
with $U^N_0=I_N$, where $W_t$ is a standard Brownian motion on $\mathfrak{u}(N)$
(for example, take $\{\xi_k\}_{k=0}^{N^2-1}$ an orthonormal basis of $\mathfrak{u}(N)$ with respect to the given inner product and $W^N_t=\sum_{j=0}^{N^2-1}b^j_t\xi_j,$ where the
$b^j_t$ are independent standard Brownian motions on $\R$). This
realization of unitary Brownian motion is computationally more useful
and is mainly what will be used in the sequel. It should be noted that
another standard description of the unitary Brownian motion is via a stochastic differential equation with respect to a Hermitian Brownian motion, which results in a difference of a factor of $i$ in the diffusion coefficient. For $t>0$, let $\rho^N_t=\mathrm{Law}(U^N_t)$ denote the end point distribution of Brownian motion; $\rho^N_t$ is called the heat kernel measure on $\mathbb{U}(N)$.

Our primary object of interest is the empirical spectral measure of
unitary Brownian motion.  
A matrix $U\in\mathbb{U}(N)$ has $N$ complex eigenvalues of modulus
one which we denote by $e^{i\theta_1},\ldots,e^{i\theta_Nß}$ (repeated
according to multiplicity), and the spectral measure of $U$ is defined
to be the probability measure on the unit circle $\mathbb{S}^1$ given by
\[\mu_U:=\frac{1}{N}\sum_{j=1}^N \delta_{e^{i\theta_j}}.\]  
In particular, for $f\in C(\mathbb{S}^1)$
\[ \int_{\mathbb{S}^1} f d\mu_U = \frac{1}{N}\sum_{j=1}^N f(e^{i\theta_j}). \]
For each fixed $t>0$, $U^N_t$ is a random unitary matrix, and we denote its empirical spectral measure by $\mu^N_t:=\mu_{U^N_t}$. In \cite{Biane1997}, Biane showed that the random probability measure $\mu^N_t$ converges weakly almost surely to a deterministic probability measure, which we denote by $\nu_t$: that is, for all $f\in C(\mathbb{S}^1)$,
\[ \lim_{N\rightarrow\infty} \int_{\mathbb{S}^1} f d\mu^N_t = \int_{\mathbb{S}^1} f\,d\nu_t \text{ a.s.} \]
The measure $\nu_t$ represents in some sense the spectral distribution of a ``free unitary Brownian motion''.
For $t>0$, $\nu_t$ possesses a continuous density that is symmetric
about $1\in\mathbb{S}^1$.  When $0<t<4$, $\nu_t$ is supported on an
arc strictly contained in the circle;  for $t\ge4$, $\mathrm{supp}(\nu_t)=\mathbb{S}^1$.
The paper \cite{CDK2016} presents a nice brief summary of these and other properties of $\nu_t$ and the construction of free unitary Brownian motion.

In the present paper, we give estimates on the $L_1$-Wasserstein
distance between the empirical spectral distribution $\mu_t^N$ and its
limiting spectral measure $\nu_t$, where for probability measures
$\mu$ and $\nu$ on $\mathbb{C}$, the $L_1$-Wasserstein distance is defined by
\[ W_1(\mu,\nu):=\inf \left\{\int|x-y|\,d\pi(x,y): \pi \text{ is a coupling of } \mu \text{ and } \nu\right\}.\]
We will also make use of the equivalent dual representation of $W_1$ due to Kantorovich and Rubenstein:
\[ W_1(\mu,\nu) = \sup \left\{\int f\,d\mu -\int f\,d\nu : |f|_{L}\le1 \right\}, \]
where $|f|_L$ denotes the Lipschitz constant of $f$.

The main results of this paper are the following.
\begin{thm}\label{T:main}
	Let $\{U^N_t\}_{t\ge0}$ be a Brownian motion on $\mathbb{U}(N)$. For $t>0$, let $\mu_t^N$ denote the empirical spectral measure $U_t$ as above, and let $\overline{\mu}_t^N$ denote the ensemble-averaged spectral measure of $U^N_t$ defined by 
\[\int_{\Circle}fd\overline{\mu}^N_t:=\E\int_{\Circle}fd\mu^N_t.\]  
	Then there is a constant $C\in(0,\infty)$ such that with
        probability one, for all $N\in\mathbb{N}$ sufficiently large and $t>0$,
\[W_1(\mu^N_{t},\overline{\mu}^N_{t})\le
C\left(\frac{t}{N^2}\right)^{1/3}.\]
and, for all $N\in\mathbb{N}$ sufficiently large and $t\ge 8(\log N)^2$,
\[W_1(\mu^N_{t},\overline{\mu}^N_{t})\le \frac{C}{N^{2/3}}.\]

\end{thm}

\begin{thm}\label{T:main-nu-t}
Let $\nu_t$ be the limiting spectral measure for unitary Brownian
motion described above.  There is a constant $C\in(0,\infty)$ such that for all $N\in\mathbb{N}$ and $t>0$
	\[W_1(\overline{\mu}_t^N,\nu_t)\le
        C\min\left\{\frac{t^{2/5}\log N}{N^{2/5}},e^{-\frac{t(1+o(1))}{8\log(N)}}+\frac{1}{N}\right\}.\]
\end{thm}
One may infer from these bounds direct (a.s.)~estimates on the rate of
convergence of the empirical spectral distribution to its limiting
distribution for all sufficiently large $N$. To the authors'
knowledge, these results constitute the first known rates of
convergence for $\mu_t^N$ itself; previously the only known
convergence rates were for moments of the ensemble-averaged spectral
measure $\overline{\mu}^N_t$ \cite{CDK2016}.  

A key advantage of such rates is that they may be applied to obtain
almost sure convergence of paths of spectral measures. The following theorem gives uniform bounds on the Wasserstein distance between the empirical spectral measures and the deterministic limiting measures on compact time intervals. 

\begin{thm}\label{T:close-paths}
Let $T\ge 0$.  There are constants $c,C$ such that for all $x\ge
c\frac{T^{2/5}\log(N)}{N^{2/5}},$ 
	\begin{align*}\P\left(\sup_{0\le t\le
		T}W_1(\mu^N_t,\nu_t)>x\right)\le
  C\left(\frac{T}{x^2}+1\right)e^{-\frac{N^2x^2}{T}}.\end{align*}
In particular,  with probability one for $N$ sufficiently large 
\[\sup_{0\le t\le
  T}W_1(\mu^N_t,\nu_t)\le c\frac{T^{2/5}\log(N)}{N^{2/5}}.\]
\end{thm}

As a technical tool, we also determine rates for the convergence in time of Biane's measure to the uniform distribution on $\mathbb{S}^1$.

\begin{prop}\label{T:biane-to-uniform}
	Let $\nu_t$ denote the limiting spectral measure and $\nu$ the uniform measure on $\mathbb{S}^1$. Then there is a constant $C\in(0,\infty)$ so that for all $t\ge1$
\[W_1(\nu_t,\nu)\le Ct^{3/2}e^{-t/4}.\]
\end{prop}

The organization of the paper is as follows. In Section \ref{s.lsi},
we establish improved concentration estimates for
heat kernel measure on $\mathbb{U}(N)$ via a coupling of Brownian
motions on $\mathbb{S}^1$ and $\mathbb{SU}(N)$. These estimates are
then used in Section \ref{s.concentration} to prove Theorem \ref{T:main}. In
Section \ref{s.nut} we use Fourier and classical approximation
methods, as well as the previously mentioned coupling argument, to give bounds on the rate of convergence of the ensemble-averaged spectral measure to the limiting measure $\nu_t$ as in Theorem \ref{T:main-nu-t}. In this section, we also give the proof of Proposition \ref{T:biane-to-uniform} using similar methods. Finally, in Section \ref{s.paths}, we prove a tail bound on the metric radius of the unitary Brownian motion and a continuity result for the family of measures $\{\nu_t\}_{t>0}$, which are then both used to give the proof of Theorem \ref{T:close-paths}.

\section{A concentration inequality for heat kernel measure}
\label{s.lsi}

In this section, we will consider concentration of measure results for Lipschitz functions of the following form. Let $(X,d)$ be a metric space equipped with Borel probability measure $\rho$. Then, under some conditions, there exists $C>0$ such that, for all $r>0$ and $F:X\rightarrow \mathbb{R}$ Lipschitz with Lipschitz constant $L$ and $\mathbb{E}|F|<\infty$,
\begin{equation}
\label{e.conc}
\rho\left(\left|F-\mathbb{E}F\right|\ge r\right) \le 2e^{-r^2/L^2C}.
\end{equation}

Concentration estimates of this type are standard for heat kernel measure on a Riemannian manifold with curvature bounded below. We recall here the necessary results. Let $(M,g)$ be a complete Riemannian manifold, and let $\Delta$ denote the Laplace-Beltrami operator acting on $C^{\infty}(M)$.  We write $P_t=e^{t\Delta/2}$ to denote the heat
semigroup; that is, for $t>0$ and any sufficiently nice function $f:M\rightarrow\mathbb{R}$,
\[ P_tf(x) = \mathbb{E}[f(\xi^x_t)] = \int_M f\,d\rho^x_t \]
where $\{\xi^x_t\}_{t\ge0}$ is the Markov diffusion on $M$ started at $x$ with generator $\Delta$ (that is, $\xi^x$ is a Brownian motion on $M$) and $\rho^x_t=\mathrm{Law}(\xi^x_t)$ is the heat kernel measure. If $\mathrm{Ric}$ denotes the Ricci curvature tensor on $M$, then $\mathrm{Ric}\ge 2k$ for $k\in\mathbb{R}$ implies that for all $t>0$ the estimate (\ref{e.conc}) holds for $\rho_t$ with coefficient $C(t)=2(1-e^{-kt/2})/k$, where when $k=0$, we interpret this to be $C(t)=t$. (A typical proof is via log Sobolev estimates.) See for example Corollary 2.6 and Lemma 6.3 of \cite{Ledoux1999} (stated in the case that $k\ge0$, which is the only relevant case here). 

For small $t$ the general machinery described above leads to a sharp concentration estimate for heat kernel measure $\rho^N_t$ on $\Unitary{N}$.  For large $t$, the estimates are
no longer sharp, but we can improve them using a coupling approach inspired by one
in \cite{MM-ECP-2013}. The following lemma gives the key idea.

\begin{lemma}
\label{l.couple}
Let $b^0$ be a real-valued Brownian motion and $z_t := e^{ib_t^0/N}$,
and let $V_t$ be a Brownian motion on $\SUnitary{N}$ issued from the identity. Then $z_tV_t$ is a Brownian motion on $\mathbb{U}(N)$.
\end{lemma}

\begin{proof}
Set $Z_t:=z_tI_N$, and note that $z_t$ and $Z_t$ satisfy the stochastic differential equations
\[ dz_t = z_t \frac{idb_t^0}{N} - \frac{1}{2N^2}z_t\,dt 
	\quad \text{ and } \quad
	dZ_t = Z_t\,db_t - \frac{1}{2N^2}Z_t dt \]
where $b_t =b_t^0\xi_0$ with $\xi_0=iI_N/N$.
Let $\{\xi_j\}_{j=1}^{N^2-1}$ be an orthonormal basis of $\mathfrak{su}(N)$, and let $\{b_t^j\}_{j=1}^{N^2-1}$ be  independent real-valued Brownian motions.
Then $\tilde{W}_t= \sum_{j=1}^{N^2-1} b_t^j \xi_j$ is a Brownian motion on $\mathfrak{su}(N)$, and $V_t$ satisfies the stochastic differential equation
\[ dV_t = V_t\circ d\tilde{W}_t
	= V_t\,d\tilde{W}_t + \frac{1}{2}V_t\sum_{\xi\in\beta} \xi^2 \,dt
	= V_t\,d\tilde{W}_t - \left(\frac{N^2-1}{2N^2}\right)V_t\, dt. \]
(Here $\circ$ denotes a Stratonovich integral, which is then expressed as an It\^o integral via the usual calculus.)

Now, $\{\xi_j\}_{j=0}^{N^2-1}$ is an orthonormal basis of $\mathfrak{u}(N)$, and $z_tV_t=Z_tV_t\in \SUnitary{N}\rtimes \mathbb{U}(1)\simeq \mathbb{U}(N)$ satisfies
\begin{align*}
d(Z_tV_t) 
	&= \left(Z_t db_t - \frac{1}{2N^2} Z_t\,dt\right)V_t + Z_t\left(V_t\,d\tilde{W}_t - \left(\frac{N^2-1}{2N^2}\right)V_t\right) 
		\\
	&= Z_tV_t\,(db_t +d\tilde{W}_t) - \frac{1}{2}Z_tV_t\,dt.
\end{align*}
Since $W_t=b_t+\tilde{W}_t$ is a Brownian motion on $\mathfrak{u}(N)$, this implies that $z_tV_t$ is a Brownian motion on $\mathbb{U}(N)$. 
\end{proof}

We use this realization of the Brownian motion on $\mathbb{U}(N)$
along with concentration properties of the laws of $z_t$ and $V_t$ to
obtain sub-Gaussian concentration independent of $t$ on
$\Unitary{N}$ for large $t$.

\begin{prop}\label{T:BM-U(N)-concentration}
Let $U_t$ be distributed according to heat kernel measure on
$\Unitary{N}$, and let $F:\Unitary{N}\to\R$ be $L$-Lipschitz.  For any $t,r>0$,
\[\P(|F(U_t)-\E F(U_t)|>r)\le 2e^{-\frac{r^2}{tL^2}}.\] 
Furthermore, there is a constant $C\in(0,\infty)$ such that for all $t\ge8(\log N)^2$ and $r>0$
\[\P\left(|F(U_t)-\E F(U_t)|>r\right)\le C e^{-\frac{r^2}{4L^2}}.\]

\end{prop}

\begin{proof} 
To prove the first statement, observe that
since the Ricci curvature on $\Unitary{N}$ is nonnegative, the comments preceding Lemma \ref{l.couple} imply that the desired concentration estimate holds for 
$\rho_t^N$ with coefficient $C(t)=t$. That is, if $F:\Unitary{N}\to\R$ is $L$-Lipschitz with
with $\mathbb{E}|F|<\infty$, then 
\[\P(|F(U_t)-\E F(U_t)|>r)\le 2e^{-\frac{r^2}{tL^2}}.\] 

To prove the second statement, observe that the representation of $U_t$ in Lemma \ref{l.couple} implies that
\begin{equation}
\begin{split}
\label{E:condition-on-z}
\P\left(|F(U_t)-\E F(U_t)|>r\right) 
	&=\P\left(|F(z_tV_t)-\E F(z_tV_t)|>r\right)\\
	&\le\E\left[\P\left[\left.\Big|F(z_tV_t)-\E\big[F(z_tV_t)\big|z_t\big]\Big|>\frac{r}{2}\right|z_t\right]\right] \\
	&\qquad +\P\left(\left|\E\big[F(z_tV_t)\big|z_t\big]-\E F(z_tV_t)\right|>\frac{r}{2}\right).\end{split}\end{equation}

Now for the first term, measure concentration for $V_t$ follows again from curvature considerations:
following for example Proposition E.15 and Lemma F.27 of \cite{AGZ}, one may compute the Ricci curvature on $\SUnitary{N}$ with respect to the given inner product as
\[ \mathrm{Ric}(X,X) = \frac{1}{2}\langle X,X\rangle_N. \]
Thus, by the discussion preceding Lemma \ref{l.couple}, $\mathrm{Law}(V_t)$ on $\SUnitary{N}$ satisfies the following concentration estimate:
if $G:\SUnitary{N}\to\R$ is $L$-Lipschitz, then 
\[\P(|G(V_t)-\E G(V_t)|>r)\le 2e^{-\frac{r^2}{L^2c(t)}},\]
where $c(t):=4(1-e^{-t/4})$. 
For $z_t$ fixed, $G=F(z_t\cdot)$ is an $L$-Lipschitz function on $\SUnitary{N}$, and so  the first term of \eqref{E:condition-on-z} is bounded by $2e^{-\frac{r^2}{4L^2}}$.  

For the second term of \eqref{E:condition-on-z}, let $K=K(z_t)$ be the random variable taking values in $\{0,\ldots,N-1\}$ such that, on $\{K=k\}$, $z_t\in [e^{\frac{2\pi ik}{N}},e^{\frac{2\pi i(k+1)}{N}})$.  
Conditioning on $K$, we have 
\begin{align}
\notag\P&\left(\left|\E\big[F(z_tV_t)\big|z_t\big]-\E F(z_tV_t)\right|>\frac{r}{2}\right) \\
	&\notag=\E\left(\P\left[\left|\E\big[F(z_tV_t)\big|z_t\big]
		-\E F(z_tV_t)\right|>\frac{r}{2}\bigg|K\right]\right)\\
	&\label{e.K}\le\E\left(\P\left[\left|\E\big[F(z_tV_t)\big|z_t\big]
		-\E \big[F(z_tV_t)\big|K\big]\right|>\frac{r}{4}\bigg|K\right]\right)\\
	&\notag\qquad\qquad+\P\left(\left|\E\big[F(z_tV_t)\big|K\big]
		-\E F(z_tV_t)\right|>\frac{r}{4}\right).
\end{align}

To deal with the first term in \eqref{e.K}, let $\E_{V_t}$ denote integration over $V_t$ only, $\E_{z_t}$ integration over $z_t$ only, and let $\E_{z_t|K=k}$ denote integration over $z_t$ conditional on $K=k$.   Observe that by independence of $V_t$ and $z_t$
\begin{align*}
\left|\E\big[F(z_tV_t)\big|z_t\big]
		-\E \big[F(z_tV_t)\big|K=k\big]\right|&=\left|\E_{V_t}\left[F(z_tV_t)\right]-\E_{V_t}\E_{z_t|K=k}\left[F(z_tV_t)\right]\right|\\&\le\E_{V_t}\left|F(z_tV_t)-\E_{z_t|K=k}\left[F(z_tV_t)\right]\right| \\
	&= \int_{\SUnitary{N}} |F(z_tV) - \mathbb{E}_{z_t|K=k}[F(z_t
          V)]| \,dh^{\SUnitary{N}}_t(V),\end{align*}
where $h^{\SUnitary{N}}_t$ denotes the density of $V_t$ with respect to
Haar measure on $\SUnitary{N}$.
Now, for $V$ fixed, $F(\cdot V)$ is an $NL$-Lipschitz function on $\Circle$. So, conditioned on $K=k$, $F(z_tV)$ can only fluctuate by as much as $2\pi L$.  Thus if $\frac{r}{4}>2\pi L$, the first term is zero. For $\frac{r}{4}\le2\pi L$, we may just use the trivial bound of 1 and choose $C$ in the statement of the proposition so that $C\ge e^{(8\pi)^2/4}$.

To deal with the second term in \eqref{e.K}, note that we can replace $V_t$ with a
Haar-distributed random matrix $V$ for $t$ sufficiently large. Indeed,
letting $dV$ denote integration with respect to Haar measure on
$\SUnitary{N}$, and assuming without loss in generality that $F(I_N)=0$,
\begin{equation}
\label{eC}
\begin{split}
\left|\E\left[F(z_tV_t)-F(z_tV)\big|z_t\right]\right|&\le\int_{\SUnitary{N}}\big|F(z_tV)\big|\big|h^{\SUnitary{N}}_t(V)-1\big|dV\\&\le LN\|h^{\SUnitary{N}}_t-1\|_1,
\end{split}\end{equation}
since the diameter of $\Unitary{N}$ is $N$. A sharp estimate of the time to equilibrium of $V_t$ was proved in Theorem 1.2 of \cite{SC1994b}, from which it follows (see the discussion preceding the theorem in \cite{SC1994b}, and note that the normalization here differs by a factor of 2 from the one used there) that
\begin{equation}
\label{eD}
\|h^{\SUnitary{N}}_t-1\|_1\le e^{-\frac{t(1+o(1))}{8\log N}}.
\end{equation}
Thus if $t\ge8(\log N)^2$, replacing $V_t$ by $V$ will only affect the
constants.

Consider therefore  
\[\P\left[\Big|\E\left[F(z_tV)\big|K\right]-\E[F(z_tV)]\Big|>\frac{r}{4}\right],\]
and write $z_t=\omega_te^{\frac{2\pi i K}{N}}$, with $\omega_t$ in
the arc from $1$ to $e^{\frac{2\pi i}{N}}$.

Observe that, by Fubini's theorem and the translation invariance of Haar measure on
$\SUnitary{N}$, 
\[\E[F(z_tV)]=\E_{z_t}\E_V[F(\omega_te^{\frac{2\pi i
    K}{N}}V)]=\E_{z_t}\E_V[F(\omega_tV)]=\E_V\E_{z_t}[F(\omega_tV)],\]
and similarly 
\begin{align*}
\E[F(z_tV)|K=k]
	&=\E_{z_t|K=k}\E_V[F(\omega_te^{\frac{2\pi i K}{N}}V)] \\
	&=\E_{z_t|K=k}\E_V[F(\omega_tV)]
	=\E_V \E_{z_t|K=k}[F(\omega_tV)].
\end{align*}
Thus 
\[\Big|\E\left[F(z_tV)\big|K\right]-\E[F(z_tV)]\Big|\le\E_V\Big|\E_{z_t}[F(\omega_tV)]-\E_{z_t|K}[F(\omega_tV)]\Big|\le2\pi
L,\]
where we have used again that for fixed $V$, $F(\omega V)$ is an
$NL$-Lipschitz function of $\omega$, and here $\omega$ lies within an arc of length $\frac{2\pi}{N}$. The estimate now follows as in the first term.
\end{proof}

\section{Concentration of $\mu_t^N$}
\label{s.concentration}

Armed with the concentration inequality for heat kernel measure, the
proof of Theorem \ref{T:main} is an application of the program laid out in
\cite{MM2016} for estimating the Wasserstein distance between the
empirical spectral measure of a random matrix and the ensemble
average, in the presence of measure concentration. Since it is relatively brief, we include the detailed argument here for completeness.

The first step is to bound the ``average distance to average'' $\E
W_1(\mu^N_t,\overline{\mu}^N_t)$ as follows.

\begin{prop}\label{T:avg-to-avg}
There is a constant $c\in(0,\infty)$ such that for all $N\in\mathbb{N}$ and $t>0$
\[\E W_1(\mu^N_t,\overline{\mu}^N_t)\le
c\left(\frac{t}{N^2}\right)^{1/3},\]
and for all $N\in\mathbb{N}$ and  $t\ge8(\log N)^2$
\[\E W_1(\mu^N_t,\overline{\mu}^N_t)\le
\frac{c}{N^{2/3}}.\]
\end{prop}

\begin{proof}We will give the proof of the first statement only, which applies the first half of Proposition \ref{T:BM-U(N)-concentration}; the proof of the second statement is identical using only instead the second half of Proposition \ref{T:BM-U(N)-concentration}.

Recall that
\[W_1(\mu^N_t,\overline{\mu}^N_t)=\sup_{|f|_L\le 1}\left(\int fd\mu^N_t-\int
fd\overline{\mu}^N_t\right),\]
where $|f|_L\le 1$.
That is, our task is to estimate the expected supremum of the centered stochastic
process $\{X_f\}_{|f|_L\le 1}$, with 
\[X_f:=\int fd\mu^N_t-\int
fd\overline{\mu}^N_t=\int fd\mu^N_t-\E\int
fd\mu^N_t.\]
Note that without loss we may choose the indexing set to be 1-Lipschitz functions
on the circle with $f(1)=0$; write $Lip_0(1)$ for the set of all such functions.
Now, if $f$ is a fixed Lipschitz function and $\mu_U$ denotes the spectral measure of $U$, then 
\[U\longmapsto \left(\int fd\mu_U-\int fd\overline{\mu}^N_t\right)\]
is $\frac{|f|_L}{N}$-Lipschitz (see Lemma 2.3 of \cite{MM-PTRF-2013},
and note the different normalization of the metric on matrices),
and so by Proposition 
\ref{T:BM-U(N)-concentration}, 
\[\Prob\left(|X_f-X_g|>x\right)=\Prob\left(|X_{f-g}|>x\right)\le 2e^{-\frac{N^2x^2}{t|f-g|_L^2}}.\]
That is, the stochastic process $\{X_f\}_{f\in Lip_0(1)}$ satisfies a
sub-Gaussian increment condition.

Now, if $\{X_v\}_{\|v\|=1}$ is a centered stochastic process indexed
by the unit ball of a \emph{finite-dimensional} normed space $V$, and
$\{X_v\}$ satisfies the increment condition 
\[\Prob(|X_u-X_v|>x)\le a e^{-\frac{x^2}{K^2\|u-v\|^2}}\]
for each $x>0$, then it is a consequence of Dudley's entropy bound
(see \cite{MM2016} for a detailed proof) that
\begin{equation}\label{E:fd-sup-bound}\E\left(\sup_{\|v\|=1}X_v\right)\le aK\sqrt{\dim V}.\end{equation}

The index set $Lip_0(1)$ is the unit ball of an infinite-dimensional
normed space, but Lipschitz test functions may be approximated by
piecewise linear functions coming from a finite-dimensional space.
Specifically, for $m\in\N$, let $A_0^{(m)}$ be the set of
$f:[0,2\pi]\to\R$ such that 
\begin{itemize}
\item $f(0)=f(2\pi)=0$,
\item $|f|_L\le 1$, and 
\item $f$ is piecewise linear, with changes in slope occurring only at
  the values
  $\frac{2\pi k}{m}$, $1\le k\le m-1$.  
\end{itemize}

For any $f\in Lip_0(1)$, there is $f^{(m)}\in
A^{(m)}_0$ such that $\|f-f^{(m)}\|_\infty\le \frac{\pi}{m}$, and so
\[\left|X_f-X_{f^{(m)}}\right|=\left|\int(f-f^{(m)})d\mu^N_t-\int(f-f^{(m)})d\overline{\mu}^N_t\right|\le\frac{2\pi}{m}.\]
The space of functions for which $A^{(m)}_0$ is the unit ball is
$(m-1)$-dimensional, and so it follows from \eqref{E:fd-sup-bound}
that
\begin{align*}
\E\left(\sup_{f\in Lip_0(1)}X_f\right)
	&\le \frac{2\pi}{m}+\E\left(\sup_{f\in A^{(m)}_0}X_f\right)\\
	&\le\frac{2\pi}{m}+C'\left(\frac{\sqrt{t}}{N}\right)\sqrt{m-1}.\end{align*}
Choosing $m=\left(\frac{N^2}{t}\right)^{1/3}$ completes the proof.

\end{proof}

The proof of Theorem \ref{T:main} is completed via the concentration
of $W_1(\mu^N_t,\overline{\mu}^N_t)$ about its mean, as follows.

\begin{prop} 
\label{T:W_1-concentration}
For all $t>0$, $N\in\mathbb{N}$,  and $x>0$,
\[\P\left(W_1(\mu^N_t,\overline{\mu}^N_t)>\E
  W_1(\mu^N_t,\overline{\mu}^N_t)+x\right)
	\le 2e^{-\frac{N^2x^2}{t}},\]
and there exists $C\in(0,\infty)$ such that for all $t\ge 8(\log N)^2$, $N\in\mathbb{N}$, and $x>0$,
\[\P\left(W_1(\mu^N_t,\overline{\mu}^N_t)>\E
  W_1(\mu^N_t,\overline{\mu}^N_t)+x\right)
	\le C e^{-N^2x^2/4}.\]
\end{prop}

\begin{proof}
Again, we prove only the first statement and the proof of the second is analogous.
 
Consider the mapping $F:\Unitary{N}\to\R$ given by
\[F(U)=W_1(\mu_U,\overline{\mu}^N_t),\]
where $\mu_U$ is the spectral measure of $U$ and $\overline{\mu}^N_t$ is
the ensemble-averaged empirical spectral measure of $U^N_t$ as before.
The function $F$ is a $\frac{1}{N}$-Lipschitz function of $U$ (again,
see Lemma 2.3 of \cite{MM-PTRF-2013}), and so by Proposition
\ref{T:BM-U(N)-concentration}, for all $t>0$ and all $x>0$,
\[\Prob\left(W_1(\mu^N_t,\overline{\mu}^N_t)-\E
  W_1(\mu^N_t,\overline{\mu}^N_t)>x\right)\le
2e^{-\frac{N^2x^2}{t}}.\]
\end{proof}

From the tail estimate of Proposition \ref{T:W_1-concentration}
together with Proposition \ref{T:avg-to-avg}, it follows that for any
$t,x>0$, 
\[\Prob\left(W_1(\mu^N_t,\overline{\mu}^N_t)> c\left(\frac{t}{N^2}\right)^{1/3} + x\right)\le
2e^{-\frac{N^2x^2}{t}}.\]
In particular, an
application of the Borel--Cantelli lemma with
$x_N=c\left(\frac{t}{N^2}\right)^{1/3}$ completes the proof of
the first statement of Theorem \ref{T:main}. The second statement
follows in the same way.

\section{Convergence to $\nu_t$}
\label{s.nut}

The previous section established a bound on the distance between the (random) spectral measure $\mu^N_t$ and the ensemble average $\overline{\mu}^N_t$.  The picture is completed by obtaining a rate of convergence of $\overline{\mu}^N_t$ to the limiting measure $\nu_t$.   The following is relevant for moderate $t$.

\begin{thm}\label{T:avg-to-limit} There is a constant $C\in(0,\infty)$ such that for all  $N\in\mathbb{N}$ and $t>0$
\[ W_1(\overline{\mu}^N_t,\nu_t) \le C\frac{t^{2/5}\log N}{N^{2/5}} .\]
\end{thm} 

\begin{proof}
The proof is via Fourier analysis and classical approximation theory, following the approach of Theorem 2.1 in
\cite{MM-PTRF-2013}.
The key ingredient of this proof is the bound (\ref{e.cdk}) below, which was proved in \cite{CDK2016}.

Let
\[ S_m(z) := \sum_{|k|<m} \hat{f}(k)z^k, \]
and observe that 
	\[\int z^kd\overline{\mu}^N_t=\frac{1}{N}\E[\tr(U_t^k)]\]
where $U_t$ is a Brownian motion on $\Unitary{N}$.
Given $f:\mathbb{S}^1\rightarrow\mathbb{R}$ a 1-Lipschitz function, it
is known that $|\hat{f}(k)|\le \frac{C}{k}$ for $k\ge 1$ (in fact, $C=\frac{\pi}{2}$; see, for
example, Theorem 4.6 of \cite{Katznelson2004}), and so
\begin{align*}
\left|\int S_m\,d\overline{\mu}^N_t -\int S_m\,d\nu_t\right|
	&= \left|\sum_{1\le|k|<m} \hat{f}(k)
\left(\frac{1}{N}\mathbb{E}[\mathrm{tr} (U_t^k)] - \int z^k\,d\nu_t\right) \right| \end{align*}
Now, by Theorem 1.3 of \cite{CDK2016}, for $t$ and $k$ fixed,
\begin{equation}
\label{e.cdk}
	\left|\frac{1}{N}\E[\tr(U_t^k)]-\int z^kd\nu_t\right|\le\frac{t^2k^4}{N^2}.
\end{equation}
Thus,
\begin{align*}
\left|\int S_m\,d\overline{\mu}^N_t -\int S_m\,d\nu_t\right|
	&\le C\sum_{1\le|k|<m}\frac{1}{k}\frac{t^2k^4}{N^2}
	\le C\frac{t^2m^4}{N^2}.
\end{align*}

The proof now proceeds exactly as in Theorem 2.1 of \cite{MM-PTRF-2013}. A theorem of Lebesgue implies that
\[ \|f-S_m\|_\infty \le C'\log m\left(\inf_g \|f-g\|_\infty\right) \]
where the infimum is over all trigonometric polynomials $g(z)=\sum_{|k|<m}a_kz^k$; see for example Theorem 2.2 of \cite{Rivlin1981}. Combining this with Jackson's theorem (Theorem 1.4 of the same reference) implies that $\|f-S_m\|_\infty\le C'\frac{\log m}{m}$, and thus
\begin{align*}
\left|\int f\,d\overline{\mu}^N_t-\int f\,d\nu_t\right|
	&\le \left|\int f\,d\overline{\mu}^N_t-\int S_m\,d\overline{\mu}^N_t\right| + \left|\int S_m\,d\overline{\mu}^N_t-\int S_m\,d\nu_t\right| \\
	& \qquad + \left|\int S_m\,d\nu_t-\int f\,d\nu_t\right|\\
	&\le C''\left(\frac{\log m}{m} + \frac{t^2m^4}{N^2}\right).
\end{align*}
Choosing $m=(N/t)^{2/5}$ then gives the stated bound.
\end{proof}

The bound above decays if and only if $t=o(N/((\log N)^{5/2}))$.  But for sufficiently large $t$, both $\overline{\mu}^N_t$ and $\nu_t$ are close to the uniform measure on the circle.  This is not reflected in the bound above, which gets worse for large $t$.  The following propositions treat the large $t$ case by appealing to convergence to stationarity.

\begin{prop}\label{T:avg-to-uniform}
	Let $\overline{\mu}^N_t$ denote the ensemble-averaged spectral measure of a random matrix $U_t$ distributed according to heat kernel measure on $\Unitary{N}$, and let $\nu$ denote the uniform probability measure on $\Circle$.  There are  constants $C,c\in(0,\infty)$ so that for all $N\in\mathbb{N}$ and $t>0$
\[W_1(\overline{\mu}^N_t,\nu)\le e^{-\frac{t(1+o(1))}{8\log(N)}}+\frac{2\pi}{N}.\]  
\end{prop}

\begin{proof}
First recall again that, as in the proof of Proposition \ref{T:avg-to-avg}, if $\mu_U$ denotes the spectral measure of $U$, then for fixed $f:\Circle\to\R$ with $|f|_L\le 1$, the function 
\[F(U)=\int fd\mu_U\]
is $\frac{1}{N}$-Lipschitz on $\Unitary{N}$.  Since $\nu$ is the spectral measure of a Haar-distributed random unitary matrix $U$ on $\Unitary{N}$, this means that 
\[\int fd\mu^N_t-\int fd\nu\le\frac{\|U_t-U\|_{N}}{N},\]
where $\|\cdot\|_N$ is the norm induced by the scaled inner product $\langle\cdot,\cdot\rangle_N$, and this holds for any coupling $(U_t,U)$ of heat kernel measure and Haar measure.  Taking expectation gives
\[\int fd\overline{\mu}^N_t-\int fd\nu=\E\left(\int fd\mu^N_t-\int fd\nu\right)\le\frac{\E \|U_t-U\|_{N}}{N}.\]
Taking the supremum over $f$ gives that
\[W_1(\overline{\mu}^N_t,\nu)\le\frac{\E\|U_t-U\|_{N}}{N},\]
and now taking infimum over couplings we have
\begin{equation}
\label{eB}
W_1(\overline{\mu}^N_t,\nu)\le\inf_{(U_t,U)}\frac{\E\|U_t-U\|_{N}}{N}=\frac{1}{N}W_1(U_t,U).
\end{equation}

Now consider the coupling $U_t\overset{d}{=}z_tV_t$ from Lemma
\ref{l.couple}, where $z_t=e^{ib_t^0/N}$ for $b^0_t$ a standard
Brownian motion on $\R$ and $V_t$ an independent Brownian motion on $\SUnitary{N}$
with $V_0=I_N$. One can similarly obtain Haar measure on the unitary
group from uniform measure on an interval and Haar measure on $\SUnitary{N}$: if
$z=e^{i\theta/N}$ with $\theta$ uniform in $[0,2\pi)$ and $V$ is independent of $\theta$ and
distributed according to Haar measure on $\SUnitary{N}$, then $zV$ is
distributed according to Haar measure on $\Unitary{N}$; see for example Lemma 16 of \cite{MM-ECP-2013}.  Moreover, by
the translation invariance of Haar measure, $\theta$ could also be
distributed uniformly on $[2\pi k,2\pi(k+1))$ for any $k\in\Z$, or
indeed be distributed according to any mixture of uniform measure on
such intervals, as long as the mixing measure is independent of $V$.

Given any such $z_t$, $z$, $V_t$, and $V$, 
for any $F:\Unitary{N}\to\R$ a 1-Lipschitz function, we have that 
\begin{multline}\label{E:coupled-differences}
	\big|\E F(U_t)-\E F(U)\big|=\big|\E F(z_tV_t)-\E F(zV)\big|\\
	\le\E\left|\E\left[F(z_tV_t)-F(z_tV)\Big|z_t\right]\right|+\left|\E\left[F(z_tV)-F(zV)\right]\right|\end{multline}
The first term of \eqref{E:coupled-differences} was already bounded in
the course of the proof of Proposition \ref{T:BM-U(N)-concentration}: 
\[\E\left|\E\left[F(z_tV_t)-F(z_tV)\Big|z_t\right]\right|\le Ne^{-\frac{t(1+o(1))}{8\log(N)}}.\]

To treat the second term, we may as in the proof of Proposition \ref{T:BM-U(N)-concentration} write $z_t=\omega_te^{\frac{2\pi i K}{N}}$, with $\omega_t$ in
the arc from $1$ to $e^{\frac{2\pi i}{N}}$ and
$K\in\{0,\ldots,N-1\}$, and similarly $z=\omega e^{\frac{2\pi i K}{N}}$
the second term of \eqref{E:coupled-differences} can be bounded as
\begin{align*}
\E\left[F(z_tV)-F(zV)\right]
	&=\E\left[F(\omega_te^{\frac{2\pi i
  K_t}{N}}V)\right]-\E\left[F(\omega e^{\frac{2\pi i K}{N}}V)\right] \\
	&= \E\left[F(\omega_tV)-F(\omega V)\right] 
	\le \frac{2\pi}{N}\cdot N,
\end{align*}
where the second equality follows from the independence of $V$ with
$(z,z_t)$ and Fubini's theorem, and the inequality uses the fact that,
for $V$ fixed,
$F(\omega V)$ is $N$-Lipschitz as a function of $\omega$, with
$\omega,\omega_t$ lying in an arc of length $\frac{2\pi}{N}$.

Combining this last estimate with (\ref{eB}),
(\ref{E:coupled-differences}), (\ref{eC}), and (\ref{eD}) implies that
\begin{equation*}\begin{split}
W_1(\overline{\mu}^N_t,\nu)&\le\frac{1}{N}\sup_{|F|_L\le1}\big|\E F(U_t)-\E
F(U)\big|\\&\le e^{-\frac{t(1+o(1))}{8\log N}}+\frac{2\pi}{N}.
 \end{split}\end{equation*}

\end{proof}

Finally, we compare the limiting (large $N$) measure $\nu_t$ to the uniform measure $\nu$. We restate and prove here Proposition \ref{T:biane-to-uniform}.

\begin{hack}
For $\nu_t$ and $\nu$ defined as above, there is a constant $C\in(0,\infty)$ so that for all $t\ge1$ 
\[W_1(\nu_t,\nu)\le Ct^{3/2}e^{-t/4}.\]
\end{hack}

\smallskip

Observe in particular that as $t\to\infty,$ $t^{3/2}e^{-t/4}\le
e^{-\frac{t}{8\log(N)}}$, and so Theorem \ref{T:main-nu-t} follows from Propositions \ref{T:avg-to-uniform} and \ref{T:biane-to-uniform}
together with the triangle inequality.

\begin{proof}[Proof of Proposition \ref{T:biane-to-uniform}]
The measure $\nu_t$ is symmetric, and the moments of $\nu_t$ for $k\ge 1$ are given by 
\[\int_{\Circle}z^kd\nu_t(z)=Q_k(t)e^{-\frac{kt}{2}},\]
where 
\[Q_k(t) :=\sum_{j=0}^{k-1}\frac{(-tk)^j}{(j+1)!}\binom{k-1}{j};\]
	see \cite{Biane1997}.
	As in the proof of Theorem \ref{T:avg-to-limit}, for a fixed 1-Lipschitz test function $f:\Circle\to\R$, let 
\[ S_m(z) := \sum_{|k|<m} \hat{f}(k)z^k \]
and we have that $|\hat{f}(k)|\le\frac{C}{k}$ for all $k\ge1$. Then since both $\nu_t$ and $\nu$ are probability measures on $\Circle$ and $\int_{\Circle}z^jd\nu(z)=0$ if $j\neq 0$,
\begin{align}
\notag
	\left|\int S_m(z)\,d\nu_t(z)-\int S_m(z)\,d\nu(z)\right|
	&=\left|\sum_{1\le|k|\le m}\hat{f}(k)\int z^k\,d\nu_t(z)\right| \\
\label{E:trig-poly-diffs}
	&\le C\sum_{1\le k\le m}\frac{1}{k}\big|Q_k(t)\big|e^{-\frac{kt}{2}}.
\end{align}

Let 
\[A_k(t):=Q_k(-t)=\sum_{j=0}^{k-1}\frac{(tk)^j}{(j+1)!}\binom{k-1}{j},\]
so that $|Q_k(t)|\le A_k(t)$.
Now, 
	\begin{align*}A_{k+1}(t)&=1+\sum_{j=1}^k\frac{[t(k+1)]^{j}}{(j+1)!}\binom{k}{j} \\
	&=1+tk(k+1)\sum_{j=1}^k\left(\frac{\left(1+\frac{1}{k}\right)^{j-1}}{j(j+1)}\right)\left[\frac{(tk)^{j-1}}{k(j-1)!}\binom{k}{j}\right]\end{align*} 
and note that
\[A_k(t)=\sum_{j=0}^{k-1}\frac{(tk)^j}{(j+1)!}\binom{k-1}{j}=\sum_{j=0}^{k-1}\frac{(tk)^j}{kj!}\binom{k}{j+1}=\sum_{\ell=1}^k\frac{(tk)^{\ell-1}}{k(\ell-1)!}\binom{k}{\ell}.\]
Since
$\frac{\left(1+\frac{1}{k}\right)^{\ell-1}}{\ell(\ell+1)}$
is decreasing as a function of $\ell$ on $\{1,\ldots,k\}$, it follows
that 
\[A_{k+1}(t)\le 1+\left(\frac{tk(k+1)}{2}\right)A_k(t)\le
tk(k+1)A_k(t),\]
since $t,k\ge 1$.
By induction and the fact that $A_1(t)=1$, this implies that 
\[|Q_k(t)|\le A_k(t)\le t^{k-1}k[(k-1)!]^2.\]
It now follows from
\eqref{E:trig-poly-diffs} that
\begin{multline*}\left|\int S_m(z)\,d\nu_t(z)-\int S_m(z)\,d\nu(z)\right|
	\le \sum_{k=1}^mt^{k-1}[(k-1)!]^2e^{-\frac{kt}{2}}\\
	\le
	  e^{-t/2}\sum_{k=1}^m\left(t(k-1)^2e^{-t/2}\right)^{k-1}\le
e^{-t/2}\sum_{k=1}^m\left(tm^2e^{-t/2}\right)^{k-1}.\end{multline*}
Choose
$m=\left\lfloor\frac{1}{\sqrt{2t}}e^{t/4}\right\rfloor,$ so
that $tm^2e^{-t/2}\le\frac{1}{2}$. Then 
\[\left|\int S_m(z)\,d\nu_t(z)-\int S_m(z)\,d\nu(z)\right|\le 2e^{-t/2}. \]
As in the proof of Theorem \ref{T:avg-to-limit}, we have that $\|S_m-f\|_\infty\le C'\frac{\log m}{m}$, which for the chosen value of $m$ yields
\[\|S_m-f\|_\infty\le C''t^{3/2}e^{-t/4}.\]
Combining these estimates completes the proof.
\end{proof}

\section{Convergence of paths}
\label{s.paths}
This section is devoted to the proof of Theorem \ref{T:close-paths}.  The idea is to first discretize the interval $[0,T]$ and apply the bound from
Proposition \ref{T:W_1-concentration} at the discretization points, then move from
approximation at this discrete set of points to approximation along an
entire path via a continuity property of the family of measures $\{\nu_t\}_{t>0}$.

The following tail bound is used in both parts of the argument.

\begin{prop}\label{T:BM-tail}
Let $\{U_t\}_{t\ge 0}$ denote Brownian motion in $\Unitary{N}$ with $U_0=I_N$, and let $d_g$ denote the geodesic distance on $\Unitary{N}$ induced by $\inprod{\cdot}{\cdot}_N$.  Then for all $\delta,r,s>0$,
	\[\P\left(\sup_{0<t<\delta}d_g(U_t,I_N)\ge r+2s\right)\le 16\left(1+\frac{r}{s}\right)^{N^2}e^{-\frac{r^2}{2\delta}}.\]
\end{prop}

\begin{proof}
If $d_g(U,I_n)<s,$ then by left invariance of the metric and the triangle inequality
\[d_g(U_t,I_N)=d_g(UU_t,U)\le d_g(UU_t,I_n)+s.\]
Thus, 
\begin{align*}
\P\left(\sup_{0<t<\delta}d_g(U_t,I_N)\ge
	2s+r\right) \le
	\inf_{d_g(U,I)\le s}\P\left(\sup_{0<t<\delta}d_g(UU_t,I_N)\ge
	s+r\right).\end{align*}
Applying the bound in
Equation (9.20) of \cite{Grigoryan1999} with $M=\Unitary{N}$ and $K=\overline{B\left(I_N,s\right)}$ (the closed
geodesic ball of radius $s$ about $I_N$) gives that 
	\[\inf_{d_g(U,I)\le s}\P\left(\sup_{0<t<\delta}d_g(UU_t,I_N)\ge s+r\right)\le16\frac{\vol(B\left(I_N,s+r\right))}{\vol(B(I_N,s))}e^{-\frac{r^2}{2\delta}}.\]
	Then, recalling again that $\mathrm{Ric}\ge0$ on $\Unitary{N}$, the Bishop--Gromov comparison theorem allows us to control the volume of balls in $\Unitary{N}$ by the volume of balls in $\mathbb{R}^{N^2}$ (see for example Theorem 3.16 of \cite{Grove1987}); in particular, 
\[\frac{\vol\left(B\left(I_N,s+r\right)\right)}{\vol\left(B\left(I_N,s\right)\right)}\le
\left(1+\frac{r}{s}\right)^{N^2},\]
which completes the proof.
\end{proof}

The following lemma gives the required continuity for the family of measures $\{\nu_t\}$.

\begin{lemma}  \label{T:nu-continuity}There is a constant $c$ such that for all $0<s<t$
\[W_1(\nu_t,\nu_s)\le c\sqrt{t-s}.\]

\end{lemma}

\begin{proof}
The triangle inequality for $W_1$ and Theorem \ref{T:avg-to-limit} imply that for any $N$
\begin{align*}W_1(\nu_t,\nu_s)
	&\le W_1(\nu_t,\overline{\mu}^N_t)+
W_1(\nu_s,\overline{\mu}^N_s)+
W_1(\overline{\mu}^N_t,\overline{\mu}^N_s)\\
	&\le C\frac{(t^{2/5}+s^{2/5})\log N}{N^{2/5}}+W_1(\overline{\mu}^N_t,\overline{\mu}^N_s).\end{align*}
Moreover, recall that
\[W_1(\overline{\mu}^N_t,\overline{\mu}^N_s)=\sup_{|f|_L\le
  1}\E\left[\int fd\mu_t^N-\int
  fd\mu^N_s\right]\le\frac{\E\|U_t-U_s\|_N}{N},\]
since $U\mapsto\int fd\mu_U$ is $\frac{|f|_L}{N}$-Lipschitz.
Trivially, for any $U,V\in\Unitary{N}$, $\|U-V\|_N\le
d_g(U,V)$.  So, using the stationarity of increments together with Proposition \ref{T:BM-tail} with $r=2s=\frac{c}{2}N\sqrt{t-s}$,
	\begin{align*}
		\E\|U_t-U_s\|_N=\E\|I_N-U_{t-s}\|_N
			&\le \E d_g(I_N,U_{t-s})\\
			&\le cN\sqrt{t-s}+N\P\left(d_g(I,U_{t-s})>cN\sqrt{t-s}\right)\\
	&\le cN\sqrt{t-s}+3^{N^2}Ne^{-c^2N^2/8}.\end{align*}
Choosing $c$ large enough
that $\log3+\frac{\log N}{N^2}-\frac{c^2}{8}<0$ for all $N$,
this gives that 
\[\E\|U_t-U_s\|_N\le cN\sqrt{t-s}+1\] and thus 
\[W_1(\nu_t,\nu_s)\le C\frac{(t^{2/5}+s^{2/5})\log N}{N^{2/5}} + c\sqrt{t-s}+\frac{1}{N}.\]
Since this holds for any $N$, the result follows.
\end{proof}

\begin{proof}[Proof of Theorem \ref{T:close-paths}]
Let $m\in\N$ such that $\frac{T}{m}\le 1$, and for $j=1,\ldots,m$, let $t_j:=\frac{jT}{m}$.
By Lemma \ref{T:nu-continuity}, 
\[\sup_{\substack{0\le s,t\le T\\
  |s-t|<\frac{T}{m}}}W_1(\nu_t,\nu_s)\le c\sqrt{\frac{T}{m}},\]  so
that if $x>3c\sqrt{\frac{T}{m}}$, then
	\begin{multline}\label{E:reduction-to-tjs}
		\P\left(\sup_{0\le t\le
		T}W_1(\mu^N_t,\nu_t)>x\right)\\
		\le  \P\left(\max_{1\le j\le
	m}\sup_{|t-t_j|<\frac{T}{m}}W_1(\mu^N_t,\mu^N_{t_j})>\frac{x}{3}\right)+\P\left(\max_{1\le         j\le  m}W_1(\mu^N_{t_j},\nu_{t_j})>\frac{x}{3}\right).\end{multline}		
	Using again that
$W_1(\mu^N_t,\mu^N_s)\le\frac{\|U_t-U_s\|_N}{N},$ we have that for any $A\subseteq[0,T]^2$
\begin{align*}
\P\left(\sup_{(s,t)\in A}W_1(\mu^N_t,\mu^N_s)>\frac{x}{3}\right)
	&\le\P\left(\sup_{(s,t)\in A}\|U_t-U_s\|>\frac{Nx}{3}\right)\\
	&=\P\left(\sup_{(s,t)\in A}\|I_N-U_t^{-1}U_s\|>\frac{Nx}{3}\right) \\
	&=\P\left(\sup_{(s,t)\in A}\|I_N-U_{t-s}\|>\frac{Nx}{3}\right),
\end{align*}
where the first equality is because $U_t\in\Unitary{N}$ and the second
is by the stationarity of the increments of Brownian motion.
It follows from this and \eqref{E:reduction-to-tjs} that
	\begin{multline*}\P\left(\sup_{0\le t\le
		T}W_1(\mu^N_t,\nu_t)>x\right) \\
	\le  m\P\left(\sup_{|t|<\frac{T}{m}}\|I_N-U_t\|>\frac{Nx}{3}\right)+m\max_{1\le         j\le  m}\P\left(W_1(\mu^N_{t_j},\nu_{t_j})>\frac{x}{3}\right).\end{multline*}
Applying Proposition \ref{T:BM-tail} to the first term with $2s=r=\frac{Nx}{6}$ gives that
\begin{align*}
\P\left(\sup_{|t|<\frac{T}{m}}\|I_N-U_t\|_N>\frac{Nx}{3}\right)
		&\le \P\left(\sup_{|t|<\frac{T}{m}}d_g(U_t,I_N)>\frac{Nx}{3}\right) \\
		&\le 3^{N^2}e^{-\frac{N^2x^2m}{72T}}. 
\end{align*}
For the second term, applying the
estimate following Proposition \ref{T:W_1-concentration} together with
Theorem \ref{T:avg-to-limit}, if $x\ge 3C\frac{T^{2/5}\log(N)}{N^{2/5}}>
6c\left(\frac{T}{N^2}\right)^{1/3}$, then 
\[\max_{1\le         j\le
	m}\P\left(W_1(\mu^N_{t_j},\nu_{t_j})>\frac{x}{3}\right) \le 2e^{-\frac{N^2x^2}{T}}.\]

	We thus have that, for any $m\in\mathbb{N}$ such that $\frac{T}{m}\le1$ and $x\ge 3C\frac{T^{2/5}\log(N)}{N^{2/5}}$,
	\begin{align*}\P\left(\sup_{0\le t\le
		T}W_1(\mu^N_t,\nu_t)>x\right)&
\le  m3^{N^2}e^{-\frac{N^2x^2m}{72T}}+2me^{-\frac{N^2x^2}{T}}.\end{align*}
Choosing $m=\left\lceil72\left(\frac{T\log 3}{x^2}+1\right)\right\rceil$ completes
the proof of the first claim; the second follows by choosing
	$x=3C\frac{T^{2/5}\log(N)}{N^{2/5}}$ and applying the Borel--Cantelli lemma.

\end{proof}

\bibliographystyle{abbrv}
\bibliography{bm}

\end{document}